\documentclass[11pt,letterpaper]{article}
\usepackage[left=1.2in, right=1.2in, bottom = 1.2in, top = 1.2in]{geometry}
\usepackage[utf8]{inputenc}
\usepackage[T1]{fontenc}
\usepackage[english]{babel}
\usepackage{amsmath}
\usepackage{amsfonts}
\usepackage{amssymb}
\usepackage{amsthm}
\usepackage{listings}
\usepackage{xcolor}
\usepackage{cite}
\usepackage{xypic}
\usepackage{pgf,tikz,pgfplots}
\usepackage{subcaption,ragged2e}
\pgfplotsset{compat=1.15}
\pgfplotsset{ticks = none}
\usepackage{mathrsfs}
\usetikzlibrary{arrows}
\usepackage{verbatim}
\usepackage{mathtools}
\usepackage{multicol} 
\usepackage{hyperref}
\usepackage{comment}

\usepackage{changes}

\usepackage{pdflscape}
\usepackage{afterpage}
\usepackage{geometry}

\newcommand*{\N}{\mathbb{N}}
\newcommand*{\Z}{\mathbb{Z}}

\newcommand*{\R}{\mathbb{R}}

\renewcommand{\phi}{\varphi}
\renewcommand{\epsilon}{\varepsilon}
\renewcommand{\theta}{\vartheta}

\DeclareMathOperator{\id}{id}

\newtheorem{prop}{Proposition}[section]

\newtheorem{thm}{Theorem}[section]
\newtheorem{cor}{Corollary}[section]
\title{Scaling limit of the sandpile identity element on the Sierpi\'nski gasket\\
\textit{\small In memory of Robert Strichartz}}

\newcommand*{\idel}{\mathsf{id}}

\newcommand*{\SG}{\mathsf{SG}}

\newcommand{\constcont}{\begin{tikzpicture}\fill[black] (0,0) -- (1/6,0)  -- (1/12,1.732/12);\end{tikzpicture}}
\newcommand{\idconst}{4}

\author{Robin Kaiser and Ecaterina Sava-Huss}
\begin{document}
\maketitle
\begin{abstract}
We investigate the identity element of the sandpile group on finite approximations of the Sierpi\'nski gasket with normal boundary conditions and show that the sequence of piecewise constant continuations of the identity elements on $\SG_n$ converges  in the weak* sense to the constant function with value $4$ on the Sierpi\'nski gasket $\SG$. We then generalize the proof to a wider range of functions and obtain the scaling limit for the identity elements with different choices of sink vertices.
\end{abstract}

\section{Introduction}
The {\it Abelian sandpile model} was first introduced on lattices as a model of {\it self-organized criticality} by Bak, Tang and Wiesenfeld in \cite{btw} and later generalized to arbitrary finite graphs by Dhar in \cite{dhar-burning}, where the term Abelian sandpile originates from. As the model and questions surrounding it combines different fields of mathematics - such as group theory and algebra, stochastics or functional analysis - it has meanwhile gathered the attention of many researchers from mathematics and physics, and progress has been made on several related questions; see \cite{crit-exp-z^d,inf-vol-limit-sandpile,heights-on-zd} for a selection of examples. However, many problems and conjectures are lacking mathematical proofs, even when the underlying state space is the Euclidean lattice.

In the Abelian sandpile model particles are placed on the vertices of a graph and if the number of particles at a vertex exceeds a certain threshold, the particles topple and are distributed evenly among the neighbors of that vertex. This explains the namesake of the model, as the particles can be seen as sand that collapses if the pile gets too large. We introduce randomness in this model, by randomly adding particles to the vertices and then stabilizing. This is known as the {\it Abelian sandpile Markov chain}, which has the interesting property that its set of recurrent states has the structure of an Abelian group, called the {\it sandpile group or the critical group}, and the group operation is given by componentwise addition followed by stabilization. See \cite{abelian-sandpile-desc} for an introduction to  the Abelian sandpile model.

One of the interesting features of this model is the appearance of fractal patterns in sandpiles, which led Robert Strichartz together with his research group at Cornell University, and especially during the summer  research programs (REU and SPUR) to tackle the model on fractal graphs, and in particular on the
Sierpi\'nski gasket. This line of research started around 2015 when the second author of the current paper was a postdoc at Cornell University, and fruitful discussions with Robert led to a sequence of works, such as
the limiting shape of the single source model in \cite{chen-sandpile-limit-shape}, the structure of the sandpile group in \cite{sandpile-group-gasket}, the calculation of the height probabilities in \cite{height-pb}. 
In one of the undergraduate research papers \cite{strichartz-undergrad} co-authored by Robert Strichartz,  the identity element of the sandpile group on some special fractal graphs was investigated, and a graphical conjecture/question about the structure of  the  identity element of the sandpile group with normal boundary conditions on finite approximation graphs of the Sierpi\'nski gasket  was stated; this question and similar ones have  been answered in \cite{sandpile-group-gasket}. The identity element of the sandpile group has been investigated in a series of works \cite{dhar-asm-1995,caracciolo-identity-2008, borgne-identity-2002} and in \cite{Pegden_2013} the authors focus on the scaling limit of the single-source sandpile model  on $\Z^2$.
Other structural properties of the sandpile identity element, and of the sandpile group itself on finite approximation graphs $\SG_n$, $n\in\N$ of the Sierpi\'nski gasket $\SG$ might bring us closer to tackling the conjectures on the critical exponents for Abelian sandpiles on fractal graphs posed by the physicists in \cite{physics1,physics3} more than two decades ago. The current work focuses on the scaling limit of  the sequence $(\idel_n)_{n\in\N}$ of identity elements of the sandpile groups on $\SG_n$.
Before stating our main result, we refer the reader to Section \ref{sec:prel}, where the level $n$ approximation graphs $\SG_n$, the Sierpi\'nski fractal $\SG$ as a limiting object of the graphs $\SG_n$, the sandpile group with normal boundary conditions and its identity $\idel_n$ on $\SG_n$, and the piecewise constant continuation  of $\idel_n$ (defined on $\SG_n$) to $\idel_n^{\constcont}$ (defined on $\SG$), are introduced. Our main result is the following.

\begin{thm}\label{thm:scaling-limit-id}
For $n\in\N$, let $\idel_n$ be the identity element of the sandpile group of $\SG_n$ with normal boundary conditions, and 
$\idel_n^{\constcont}$  the piecewise constant continuation  to the gasket $\SG$ of $\idel_n$. 
Then the sequence $(\idel_n^{\constcont})_{n\in\N}$ converges in the weak* sense to the function $\idel_\SG$ given by
\begin{align*}
    \idel_\SG:\SG\rightarrow\R:x\mapsto \idconst.
\end{align*}
\end{thm}

\section{Preliminaries}\label{sec:prel}
\subsection{Abelian sandpile model}

We consider an undirected, connected and finite graph $G=(V\cup\{s\},E)$, where the special vertex $s$ is known as the {\it sink vertex}. Note that it is allowed for $G$ to have loops and even multiple edges between two vertices. In fact, in the case of the Sierpi\'{n}ski gasket, every corner will be connected to the sink by two edges. For $v\in V\cup\{s\}$, the degree $\deg_G(v)$ of $v$ is the number of neighbors of $v$ in $G$. For vertices $v,w$ we write $w\sim_G v$ if they are neighbors in $G$. A sandpile is a function $\eta:V\rightarrow\N=\{0,1\ldots\}$. We call a sandpile $\eta$ stable if for every $v\in V$ we have $\eta(v)<\deg_G(v)$ and unstable if it is not stable, that is it exists at least one vertex $v$ with $\eta(v)\geq\deg_G(v)$. We define the toppling at $v$ by
\begin{align*}
    T_v\eta=\eta-\Delta_G\delta_v,
\end{align*}
where $\delta_v:V\rightarrow\Z$ is the function that is $0$ everywhere except at $v$, where it takes the value $1$, and $\Delta_G\in\Z^{V\times V}$ is the graph Laplacian given by
\begin{small}\begin{align*}
    \Delta_G(x,y)=\begin{cases}
    \deg_G(x),&x=y\\
    -1,&x\sim_G y\\
    0,&\text{else}
    \end{cases}.
\end{align*}\end{small}
We say that the toppling at $v$ is legal if $\eta(v)\geq\deg_G(v)$. Note that when a neighbor of the sink vertex $s$ topples, particles sent along the edges to the sink will disappear from the system, and the role of the sink is to collect the excess mass. Given an unstable sandpile $\eta$, there always exists a sequence of legal topplings at $v_1,...,v_n$ such that $T_{v_1}...T_{v_n}\eta$ is stable. We then define the stabilization of $\eta$ as
\begin{align*}
    \eta^\circ=T_{v_1}...T_{v_n}\eta.
\end{align*}
Notice that the stabilization is unique. We denote by $\eta_{\max}$ the sandpile given by $\eta_{\max}(v)=\deg_G(v)-1$ for every vertex $v$. A sandpile $\eta$ is called \textit{recurrent} if there exists a sandpile $\zeta:V\rightarrow\N$ such that $\eta=(\eta_{\max}+\zeta)^\circ$ and we denote by $\mathcal{R}_G$ the set of all recurrent sandpiles on $G$. The set $\mathcal{R}_G$  together with the operation $\oplus$ given by
\begin{align*}
    \eta\oplus\zeta=(\eta+\zeta)^\circ
\end{align*}
forms and Abelian group called the \emph{sandpile group} of $G$, and the identity sandpile in $\mathcal{R}_G$ is
denoted by $\id_G$. Intuitively the operation $\oplus$ consists of stacking the particles of $\eta$ and $\zeta$ on top of each other, and then stabilizing the new sandpile configuration.

\begin{figure}
    \centering
    \begin{subfigure}[t]{0.3\linewidth}
        \centering
        \resizebox{\linewidth}{!}{
        \begin{tikzpicture}[baseline=9ex]
        \node[shape=circle,draw=black] (A) at (0,0) {};
        \node[shape=circle,draw=black] (B) at (4,0) {};
        \node[shape=circle,draw=black] (E) at (2,1.73*2) {};
        
        \node[shape=circle,draw=none] (A1) at (-1,0) {};
        \node[shape=circle,draw=none] (A2) at (-1/2,-1.73/2) {};
        \node[shape=circle,draw=none] (B1) at (5,0) {};
        \node[shape=circle,draw=none] (B2) at (4+1/2,-1.73/2) {};
        \node[shape=circle,draw=none] (C1) at (3/2,1.73*2+1.73/2) {};
        \node[shape=circle,draw=none] (C2) at (5/2,1.73*2+1.73/2) {};
        
        \path [-] (A) edge node[left] {} (B);
        \path [-] (A) edge node[left] {} (E);
        \path [-] (B) edge node[left] {} (E);
        \end{tikzpicture}
        }
    \end{subfigure}
    \begin{subfigure}[t]{0.3\linewidth}
        \centering
        \resizebox{\linewidth}{!}{
        \begin{tikzpicture}[baseline=9ex]
        \node[shape=circle,draw=black] (A) at (0,0) {};
        \node[shape=circle,draw=black] (B) at (2,0) {};
        \node[shape=circle,draw=black] (C) at (4,0) {};
        \node[shape=circle,draw=black] (D) at (1,1.73) {};
        \node[shape=circle,draw=black] (E) at (3,1.73) {};
        \node[shape=circle,draw=black] (F) at (2,1.73*2) {} ;
        
        \node[shape=circle,draw=none] (A1) at (-1,0) {};
        \node[shape=circle,draw=none] (A2) at (-1/2,-1.73/2) {};
        \node[shape=circle,draw=none] (B1) at (5,0) {};
        \node[shape=circle,draw=none] (B2) at (4+1/2,-1.73/2) {};
        \node[shape=circle,draw=none] (C1) at (3/2,1.73*2+1.73/2) {};
        \node[shape=circle,draw=none] (C2) at (5/2,1.73*2+1.73/2) {};
        
        \path [-] (A) edge node[left] {} (B);
        \path [-] (A) edge node[left] {} (D);
        \path [-] (B) edge node[left] {} (D);
        \path [-] (B) edge node[left] {} (C);
        \path [-] (E) edge node[left] {} (B);
        \path [-] (E) edge node[left] {} (C);
        \path [-] (F) edge node[left] {} (E);
        \path [-] (F) edge node[left] {} (D);
        \path [-] (E) edge node[left] {} (D);
        \end{tikzpicture}}
    \end{subfigure}
    \begin{subfigure}[t]{0.3\linewidth}
        \centering
        \resizebox{\linewidth}{!}{
        \begin{tikzpicture}[baseline=9ex]
        \node[shape=circle,draw=black] (A) at (0,0) {};
        \node[shape=circle,draw=black] (B) at (2,0) {};
        \node[shape=circle,draw=black] (C) at (4,0) {};
        \node[shape=circle,draw=black] (D) at (1,1.73) {};
        \node[shape=circle,draw=black] (E) at (3,1.73) {};
        \node[shape=circle,draw=black] (F) at (2,1.73*2) {} ;
        \node[shape=circle,draw=black] (G) at (1,0) {};
        \node[shape=circle,draw=black] (H) at (3,0) {};
        \node[shape=circle,draw=black] (I) at (1/2,1.73/2) {};
        \node[shape=circle,draw=black] (J) at (3/2,1.73/2) {};
        \node[shape=circle,draw=black] (K) at (5/2,1.73/2) {};
        \node[shape=circle,draw=black] (L) at (7/2,1.73/2) {};
        \node[shape=circle,draw=black] (M) at (3/2,1.73/2+1.73) {};
        \node[shape=circle,draw=black] (N) at (5/2,1.73/2+1.73) {};
        \node[shape=circle,draw=black] (O) at (2,1.73) {};
        
        \node[shape=circle,draw=none] (A1) at (-1,0) {};
        \node[shape=circle,draw=none] (A2) at (-1/2,-1.73/2) {};
        \node[shape=circle,draw=none] (B1) at (5,0) {};
        \node[shape=circle,draw=none] (B2) at (4+1/2,-1.73/2) {};
        \node[shape=circle,draw=none] (C1) at (3/2,1.73*2+1.73/2) {};
        \node[shape=circle,draw=none] (C2) at (5/2,1.73*2+1.73/2) {};
        
        \path [-] (A) edge node[left] {} (G);
        \path [-] (B) edge node[left] {} (G);
        \path [-] (A) edge node[left] {} (I);
        \path [-] (D) edge node[left] {} (I);
        \path [-] (B) edge node[left] {} (J);
        \path [-] (D) edge node[left] {} (J);
        \path [-] (I) edge node[left] {} (J);
        \path [-] (I) edge node[left] {} (G);
        \path [-] (J) edge node[left] {} (G);
        
        \path [-] (B) edge node[left] {} (H);
        \path [-] (C) edge node[left] {} (H);
        \path [-] (E) edge node[left] {} (K);
        \path [-] (E) edge node[left] {} (L);
        \path [-] (B) edge node[left] {} (K);
        \path [-] (C) edge node[left] {} (L);
        \path [-] (H) edge node[left] {} (K);
        \path [-] (H) edge node[left] {} (L);
        \path [-] (L) edge node[left] {} (K);
        
        \path [-] (F) edge node[left] {} (M);
        \path [-] (F) edge node[left] {} (N);
        \path [-] (E) edge node[left] {} (N);
        \path [-] (E) edge node[left] {} (O);
        \path [-] (D) edge node[left] {} (O);
        \path [-] (D) edge node[left] {} (M);
        \path [-] (M) edge node[left] {} (N);
        \path [-] (O) edge node[left] {} (N);
        \path [-] (O) edge node[left] {} (M);
        \end{tikzpicture}}
    \end{subfigure}
    \caption{The graphs $\SG_0$, $\SG_1$ and $\SG_2$.}
    \label{fig:sierp_constr}
\end{figure}
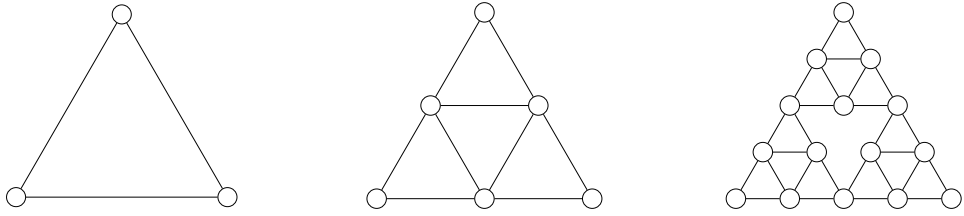

\newpage
\subsection{Sierpi\'nski gasket}
We introduce here the Sierpi\'nski gasket $\SG$ as well as its finite approximation graphs $\SG_n$. Consider for $i\in\{1,2,3\}$ the three functions $\psi_i:\R^2\rightarrow\R^2$ given by
\begin{align*}
    \psi_i(x)=\frac{1}{2}(x-u_i)+u_i,
\end{align*}
and $u_1=(0,0)$, $u_2=(1,0)$ and $u_3=\frac{1}{2}(1,\sqrt{3})$. The Sierpi\'nski gasket is the unique non-empty compact set $\SG$ such that $\SG=\cup_{i=1}^3\psi_i(\SG)$. Its discrete approximation graphs $\SG_n$, for $n\in\N$ are  defined as follows.
Denote by $\SG_0$ the complete graph with vertex set $\{u_1,u_2,u_3\}$. For $n\geq 1$, given $\SG_{n-1}$, we define $\SG_n$ as the level $n$ approximation graph of the Sierpi\'nski gasket by $\SG_n=\cup_{i=1}^3\psi_i(\SG_{n-1})$. See Figure \ref{fig:sierp_constr} for an illustration of the first three iterations of the gasket graphs.

\textbf{Piecewise constant continuation.} Given a function $f:\SG_n\rightarrow\R$, we define the piecewise constant continuation of $f$ to $\SG$ to be the function $f^{\constcont}:\SG\rightarrow\R$ defined by
\begin{align*}
    f^{\constcont}(x)=f(B(x,2^{-(n+1)})\cap\SG_n),
\end{align*}
where $B\big(x,2^{-(n+1)}\big)$ is the ball around $x$ with radius $2^{-(n+1)}$ in $\SG$, in the Euclidean distance. If $|B(x,2^{-(n+1)})\cap\SG_n|=1$, then $f(B(x,2^{-(n+1)})\cap\SG_n)$ is the evaluation of $f$ at this single element. If $|B(x,2^{-(n+1)})\cap\SG_n|=2$, then  $f(B(x,2^{-(n+1)})\cap\SG_n)$ is the evaluation of the element in $B(x,2^{-(n+1)})\cap\SG_n$ with smaller $x$-coordinate. The cases $|B(x,2^{-(n+1)})\cap\SG_n|>2$ and $|B(x,2^{-(n+1)})\cap\SG_n|=0$ are not possible.

\textbf{Hausdorff measure.} Given a word $w$ of finite length over the alphabet $\{1,2,3\}$, that is $w\in\cup_{m=1}^\infty \{1,2,3\}^m$, define $\psi_w$ by
\begin{align*}
    \psi_w=\psi_{w_n}\circ...\circ\psi_{w_1}.
\end{align*}
The Hausdorff measure $\mu$ on $\SG$ is the unique self-similar measure on $\SG$ with $\mu(\SG)=1$ and
\begin{align*}
    \mu(\psi_w(\SG))=3^{-|w|},
\end{align*}
for every $w\in\cup_{m=1}^\infty \{1,2,3\}^m$, where $|w|$ is the length of $w$. The measurable sets of $\SG$ are then given by the sigma algebra generated by
\begin{align*}
    \{\psi_w(\SG):w\in\cup_{m=1}^\infty \{1,2,3\}^m\}.
\end{align*}
See \cite{existence-mu} for the existence and uniqueness of $\mu$.

\paragraph{Weak* convergence.}
 Given a sequence of functions $(f_n)_{n\in\N}$ with $f_n:\SG\rightarrow\R$ and $f_n\in L^\infty(\mu)$ for every $n\in\N$, as well as a function $f:\SG\rightarrow\R$ with $f\in L^\infty(\mu)$ we define weak* convergence in $L^\infty(\mu)$ as following. The sequence $(f_n)_{n\in\N}$ converges in the weak* sense to $f$ in $L^\infty(\mu)$ if for every integrable function $g\in L^1(\mu)$ we have
\begin{align}\label{eq:conv-in-mean}
\int_{\SG}f_n\cdot g\ d\mu\xrightarrow{n\rightarrow\infty}\int_{\SG}f\cdot g\ d\mu,
\end{align}
and we write shortly $f_n\xrightarrow{w*}f$. By interpreting functions from $L^\infty(\mu)$ as linear, continuous functions over $L^1(\mu)$, weak* convergence is sometimes referred to as pointwise convergence of the linear functionals. While proving \eqref{eq:conv-in-mean} for all integrable functions $g$ is not always very easy, it suffices for our purposes to consider integrals over small copies of the gasket, that is sets of the form $\psi_w(\SG)$ for some $w\in\cup_{m=1}^\infty \{1,2,3\}^m$. To see this, let us first notice that it is actually enough to consider only simple functions $g$. Assume that  \eqref{eq:conv-in-mean} holds for all simple functions, and for all $n\in\N$ we have $||f_n||_\infty<4$ and $||f||_\infty < 4$. The boundedness assumption is reasonable in our case, as the stable sandpile configurations we consider throughout the paper have to be smaller than the degree $4$ of the graph by definition. Then given an integrable function $g$ and $\varepsilon>0$, one can find a simple function $h$ such that
\begin{align*}
||g-h||_1<\varepsilon/16.
\end{align*}
Furthermore, let us choose $N\in\N$ such that for all $n\geq N$ we have
\begin{align*}
\Big|\int_{\SG}(f_n-f)\cdot h\ d\mu\Big|<\varepsilon/2.
\end{align*}
We then have for all $n\geq N$
\begin{align*}
\Big|\int_{\SG}(f_n-f)\cdot g\ d\mu\Big|&=\Big|\int_{\SG}(f_n-f)\cdot (g-h)\ d\mu+\int_{\SG}(f_n-f)\cdot h\ d\mu\Big|\\
&\leq \int_{\SG}|f_n-f|\cdot |g-h|\ d\mu+\Big|\int_{\SG}(f_n-f)\cdot h\ d\mu\Big|\\
&\leq 8||g-h||_1+\Big|\int_{\SG}(f_n-f)\cdot h\ d\mu\Big|\\
&\leq\varepsilon/2+\varepsilon/2=\varepsilon,
\end{align*}
which shows that \eqref{eq:conv-in-mean} holds also for $g$. Finally, it suffices to consider only indicator functions of sets of the form $\psi_w(\SG)$ for some $w\in\cup_{m=1}^\infty \{1,2,3\}^m$, since these sets can be used to approximate any measurable set on the gasket arbitrarily close. We thus have
\begin{align}\label{eq:char-of-weak*}
f_n\xrightarrow{w*}f \Leftrightarrow \forall w\in\cup_{m=1}^\infty \{1,2,3\}^m:\int_{\psi_w(\SG)}f_n\ d\mu\xrightarrow{n\rightarrow\infty}\int_{\psi_w(\SG)}f\ d\mu,
\end{align}
which is what we will show for the sequence of identities on the gasket.

\begin{figure}
    \centering
    \begin{subfigure}[t]{0.3\linewidth}
        \centering
        \resizebox{\linewidth}{!}{
        \begin{tikzpicture}[baseline=9ex]
        \node[shape=circle,draw=black] (A) at (0,0) {};
        \node[shape=circle,draw=black] (B) at (4,0) {};
        \node[shape=circle,draw=black] (E) at (2,1.73*2) {};
        
        \node[shape=circle,draw=none] (A1) at (-1,0) {};
        \node[shape=circle,draw=none] (A2) at (-1/2,-1.73/2) {};
        \node[shape=circle,draw=none] (B1) at (5,0) {};
        \node[shape=circle,draw=none] (B2) at (4+1/2,-1.73/2) {};
        \node[shape=circle,draw=none] (C1) at (3/2,1.73*2+1.73/2) {};
        \node[shape=circle,draw=none] (C2) at (5/2,1.73*2+1.73/2) {};
        
        \path [-] (A) edge node[left] {} (B);
        \path [-] (A) edge node[left] {} (E);
        \path [-] (B) edge node[left] {} (E);
        
        \path [-] (A) edge node[left] {} (A1);
        \path [-] (A) edge node[left] {} (A2);
        \path [-] (B) edge node[left] {} (B1);
        \path [-] (B) edge node[left] {} (B2);
        \path [-] (E) edge node[left] {} (C1);
        \path [-] (E) edge node[left] {} (C2);
        \end{tikzpicture}
        }
    \end{subfigure}
    \begin{subfigure}[t]{0.3\linewidth}
        \centering
        \resizebox{\linewidth}{!}{
        \begin{tikzpicture}[baseline=9ex]
        \node[shape=circle,draw=black] (A) at (0,0) {};
        \node[shape=circle,draw=black] (B) at (2,0) {};
        \node[shape=circle,draw=black] (C) at (4,0) {};
        \node[shape=circle,draw=black] (D) at (1,1.73) {};
        \node[shape=circle,draw=black] (E) at (3,1.73) {};
        \node[shape=circle,draw=black] (F) at (2,1.73*2) {} ;
        
        \node[shape=circle,draw=none] (A1) at (-1,0) {};
        \node[shape=circle,draw=none] (A2) at (-1/2,-1.73/2) {};
        \node[shape=circle,draw=none] (B1) at (5,0) {};
        \node[shape=circle,draw=none] (B2) at (4+1/2,-1.73/2) {};
        \node[shape=circle,draw=none] (C1) at (3/2,1.73*2+1.73/2) {};
        \node[shape=circle,draw=none] (C2) at (5/2,1.73*2+1.73/2) {};
        
        \path [-] (A) edge node[left] {} (B);
        \path [-] (A) edge node[left] {} (D);
        \path [-] (B) edge node[left] {} (D);
        \path [-] (B) edge node[left] {} (C);
        \path [-] (E) edge node[left] {} (B);
        \path [-] (E) edge node[left] {} (C);
        \path [-] (F) edge node[left] {} (E);
        \path [-] (F) edge node[left] {} (D);
        \path [-] (E) edge node[left] {} (D);
        
        \path [-] (A) edge node[left] {} (A1);
        \path [-] (A) edge node[left] {} (A2);
        \path [-] (C) edge node[left] {} (B1);
        \path [-] (C) edge node[left] {} (B2);
        \path [-] (F) edge node[left] {} (C1);
        \path [-] (F) edge node[left] {} (C2);
        \end{tikzpicture}}
    \end{subfigure}
    \begin{subfigure}[t]{0.3\linewidth}
        \centering
        \resizebox{\linewidth}{!}{
        \begin{tikzpicture}[baseline=9ex]
        \node[shape=circle,draw=black] (A) at (0,0) {};
        \node[shape=circle,draw=black] (B) at (2,0) {};
        \node[shape=circle,draw=black] (C) at (4,0) {};
        \node[shape=circle,draw=black] (D) at (1,1.73) {};
        \node[shape=circle,draw=black] (E) at (3,1.73) {};
        \node[shape=circle,draw=black] (F) at (2,1.73*2) {} ;
        \node[shape=circle,draw=black] (G) at (1,0) {};
        \node[shape=circle,draw=black] (H) at (3,0) {};
        \node[shape=circle,draw=black] (I) at (1/2,1.73/2) {};
        \node[shape=circle,draw=black] (J) at (3/2,1.73/2) {};
        \node[shape=circle,draw=black] (K) at (5/2,1.73/2) {};
        \node[shape=circle,draw=black] (L) at (7/2,1.73/2) {};
        \node[shape=circle,draw=black] (M) at (3/2,1.73/2+1.73) {};
        \node[shape=circle,draw=black] (N) at (5/2,1.73/2+1.73) {};
        \node[shape=circle,draw=black] (O) at (2,1.73) {};
        
        \node[shape=circle,draw=none] (A1) at (-1,0) {};
        \node[shape=circle,draw=none] (A2) at (-1/2,-1.73/2) {};
        \node[shape=circle,draw=none] (B1) at (5,0) {};
        \node[shape=circle,draw=none] (B2) at (4+1/2,-1.73/2) {};
        \node[shape=circle,draw=none] (C1) at (3/2,1.73*2+1.73/2) {};
        \node[shape=circle,draw=none] (C2) at (5/2,1.73*2+1.73/2) {};
        
        \path [-] (A) edge node[left] {} (G);
        \path [-] (B) edge node[left] {} (G);
        \path [-] (A) edge node[left] {} (I);
        \path [-] (D) edge node[left] {} (I);
        \path [-] (B) edge node[left] {} (J);
        \path [-] (D) edge node[left] {} (J);
        \path [-] (I) edge node[left] {} (J);
        \path [-] (I) edge node[left] {} (G);
        \path [-] (J) edge node[left] {} (G);
        
        \path [-] (B) edge node[left] {} (H);
        \path [-] (C) edge node[left] {} (H);
        \path [-] (E) edge node[left] {} (K);
        \path [-] (E) edge node[left] {} (L);
        \path [-] (B) edge node[left] {} (K);
        \path [-] (C) edge node[left] {} (L);
        \path [-] (H) edge node[left] {} (K);
        \path [-] (H) edge node[left] {} (L);
        \path [-] (L) edge node[left] {} (K);
        
        \path [-] (F) edge node[left] {} (M);
        \path [-] (F) edge node[left] {} (N);
        \path [-] (E) edge node[left] {} (N);
        \path [-] (E) edge node[left] {} (O);
        \path [-] (D) edge node[left] {} (O);
        \path [-] (D) edge node[left] {} (M);
        \path [-] (M) edge node[left] {} (N);
        \path [-] (O) edge node[left] {} (N);
        \path [-] (O) edge node[left] {} (M);
        
        \path [-] (A) edge node[left] {} (A1);
        \path [-] (A) edge node[left] {} (A2);
        \path [-] (C) edge node[left] {} (B1);
        \path [-] (C) edge node[left] {} (B2);
        \path [-] (F) edge node[left] {} (C1);
        \path [-] (F) edge node[left] {} (C2);
        \end{tikzpicture}}
    \end{subfigure}
    \caption{The first three iterations $\SG_0,\SG_1,\SG_2$ of the Sierpi\'nski gasket with normal boundary conditions.}
    \label{fig:sierp_constr_normal}
\end{figure}
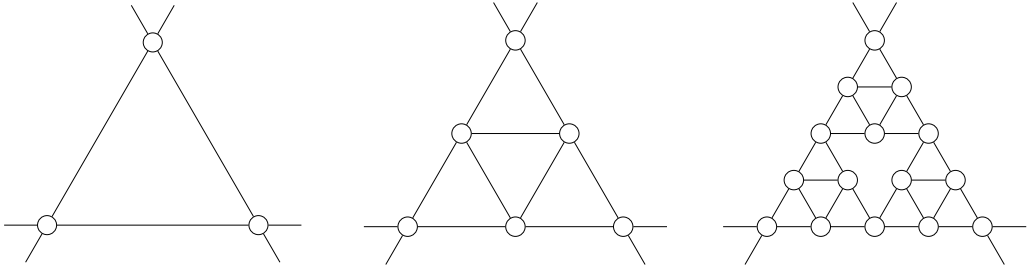

\subsection{Abelian sandpiles on finite approximation gaskets}

When investigating the Abelian sandpile model on the graphs $\SG_n$, we consider the so-called {\it normal boundary conditions}. That is, on $\SG_n$ we introduce a new vertex $s$ that is connected to the three corner vertices by $2$ edges respectively, see Figure \ref{fig:sierp_constr_normal}. In \cite{sandpile-group-gasket} the sequence of identity elements $\id_{\SG_n}$ - from now on denoted by $\idel_n$ - was determined. We recall here the description of $\id_n$ from \cite{sandpile-group-gasket}.
We define the sandpile configuration $M_1(x,y,z)$ on $\SG_1$ by setting the values at the inner vertices to be $3,3,2$ as below, and corner values $x,y,z\in\N$ are arbitrary: 
\begin{center}
\begin{tikzpicture}[baseline=7ex, scale=0.7]
    \node[shape=circle,draw=black] (A) at (0,0) {x};
    \node[shape=circle,draw=black] (B) at (2,0) {3};
    \node[shape=circle,draw=black] (C) at (4,0) {y};
    \node[shape=circle,draw=black] (D) at (1,1.73) {3};
    \node[shape=circle,draw=black] (E) at (3,1.73) {2};
    \node[shape=circle,draw=black] (F) at (2,1.73*2) {z} ;
    
    \path [-] (A) edge node[left] {} (B);
    \path [-] (A) edge node[left] {} (D);
    \path [-] (B) edge node[left] {} (D);
    \path [-] (B) edge node[left] {} (C);
    \path [-] (E) edge node[left] {} (B);
    \path [-] (E) edge node[left] {} (C);
    \path [-] (F) edge node[left] {} (E);
    \path [-] (F) edge node[left] {} (D);
    \path [-] (E) edge node[left] {} (D);
\end{tikzpicture}=
\begin{tikzpicture}[baseline=4.5ex]
    \node[shape=circle,draw=black] (A) at (0,0) {x};
    \node[shape=circle,draw=black] (B) at (2,0) {y};
    \node[shape=circle,draw=black] (D) at (1,1.73) {z};
    
    \path [-] (A) edge node[left] {} (B);
    \path [-] (A) edge node[left] {} (D);
    \path [-] (B) edge node[left] {} (D);
    
    \node[] at (1,1.73/3) {$M_1$};
\end{tikzpicture}=$\quad M_1(x,y,z)$.
\end{center}
 For $n\geq 1$, we iteratively define the sandpile configuration $M_{n+1}(x,y,z)$ with boundary values $x,y,z$ on $\SG_{n+1}$ by setting it equal to $M_n(x,3,3)$ in the lower left triangle, $M_n(3,y,2)$ in the lower right triangle and to $M_n(3,2,z)$ in the upper triangle: 
\begin{center}
\begin{tikzpicture}[baseline=6ex, scale=0.7]
    \node[shape=circle,draw=black] (A) at (0,0) {x};
    \node[shape=circle,draw=black] (B) at (2,0) {3};
    \node[shape=circle,draw=black] (C) at (4,0) {y};
    \node[shape=circle,draw=black] (D) at (1,1.73) {3};
    \node[shape=circle,draw=black] (E) at (3,1.73) {2};
    \node[shape=circle,draw=black] (F) at (2,1.73*2) {z} ;
    
    \path [-] (A) edge node[left] {} (B);
    \path [-] (A) edge node[left] {} (D);
    \path [-] (B) edge node[left] {} (D);
    \path [-] (B) edge node[left] {} (C);
    \path [-] (E) edge node[left] {} (B);
    \path [-] (E) edge node[left] {} (C);
    \path [-] (F) edge node[left] {} (E);
    \path [-] (F) edge node[left] {} (D);
    \path [-] (E) edge node[left] {} (D);
    
    \node[] at (1,1.73/3) {$M_n$};
    \node[] at (3,1.73/3) {$M_n$};
    \node[] at (2,1.73/3+1.73) {$M_n$};
\end{tikzpicture}=
\begin{tikzpicture}[baseline=4ex]
    \node[shape=circle,draw=black] (A) at (0,0) {x};
    \node[shape=circle,draw=black] (B) at (2,0) {y};
    \node[shape=circle,draw=black] (D) at (1,1.73) {z};
    
    \path [-] (A) edge node[left] {} (B);
    \path [-] (A) edge node[left] {} (D);
    \path [-] (B) edge node[left] {} (D);
    
    \node[] at (1,1.73/3) {$M_{n+1}$};
\end{tikzpicture}=$\quad M_{n+1}(x,y,z)$
\end{center}
The configurations $M_n(x,y,z):\SG_n\to\N$ with given corner values $\{x,y,z\}$ as above can be used to describe the identity elements $\idel_n$ as in \cite[Theorem 3.2]{sandpile-group-gasket}. We write $M_n(x,y,z)(v)$ for the value of the configuration $M_n(x,y,z)$ at vertex $v \in\SG_n$.
\begin{thm}[Theorem 3.2 from \cite{sandpile-group-gasket}]\label{def:id}
Denote by $M_n^+$ (respectively $M_n^-$) the sandpile configuration on $\SG_n$ obtained from $M_n$ by rotating $\SG_n$  counterclockwise (respectively clockwise) $120^\circ$. Then, for any $n\geq 1$ the identity element $\mathsf{id}_{n+1}$ of the sandpile group $(\mathcal{R}_{\SG_{n+1}},\oplus)$ of $\SG_{n+1}$ with normal boundary conditions is given by:
\begin{center}
    \begin{tikzpicture}[baseline=6ex, scale=0.7]
    \node[shape=circle,draw=black] (A) at (0,0) {2};
    \node[shape=circle,draw=black] (B) at (2,0) {2};
    \node[shape=circle,draw=black] (C) at (4,0) {2};
    \node[shape=circle,draw=black] (D) at (1,1.73) {2};
    \node[shape=circle,draw=black] (E) at (3,1.73) {2};
    \node[shape=circle,draw=black] (F) at (2,1.73*2) {2} ;
    
    \path [-] (A) edge node[left] {} (B);
    \path [-] (A) edge node[left] {} (D);
    \path [-] (B) edge node[left] {} (D);
    \path [-] (B) edge node[left] {} (C);
    \path [-] (E) edge node[left] {} (B);
    \path [-] (E) edge node[left] {} (C);
    \path [-] (F) edge node[left] {} (E);
    \path [-] (F) edge node[left] {} (D);
    \path [-] (E) edge node[left] {} (D);
    
    \node[] at (1,1.73/3) {$M_{n}$};
    \node[] at (3,1.73/3) {$M^+_{n}$};
    \node[] at (2,1.73/3+1.73) {$M^-_{n}$};
\end{tikzpicture}=
$\quad \mathsf{id}_{n+1}$.
\end{center}
\end{thm}
\section{Scaling limit of the identity elements}\label{sec:id-scaling}
\begin{figure}
    \centering
    \begin{subfigure}[t]{0.22\textwidth}
        \includegraphics[width=\linewidth]{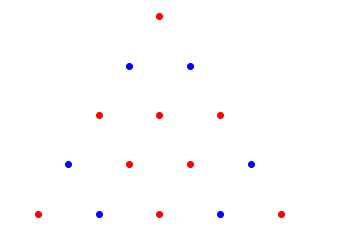}
        \caption{$\mathsf{id}_2$}
    \end{subfigure}
    \begin{subfigure}[t]{0.25\textwidth}
        \includegraphics[width=\linewidth]{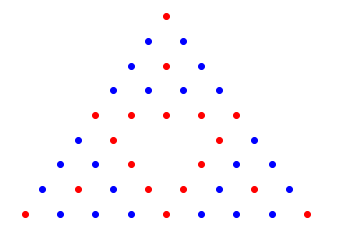}
        \caption{$\mathsf{id}_3$}
    \end{subfigure}
    \begin{subfigure}[t]{0.25\textwidth}
        \includegraphics[width=\linewidth]{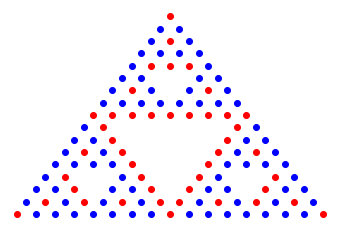}
        \caption{$\mathsf{id}_4$}
    \end{subfigure}
    \begin{subfigure}[t]{0.25\textwidth}
        \includegraphics[width=\linewidth]{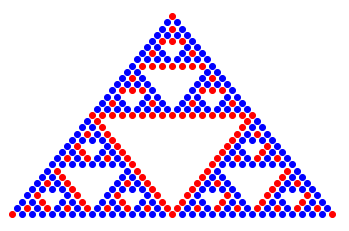}
        \caption{$\mathsf{id}_5$}
    \end{subfigure}
    \caption{The identity element on the first four levels of the Sierpi\'nski gasket. Red dots correspond to vertices with $2$ chips, whereas blue dots correspond to vertices with $3$ chips.}
    \label{fig:identity}
\end{figure}
For the proof of the scaling limit of the identity elements  two properties of the Hausdorff measure are crucial:
\begin{enumerate}
\setlength\itemsep{0em}
    \item $\mu$ is invariant under $120^\circ$ rotations.
    \item For every $m\in\N$ and every function $f:\SG\rightarrow\R$ we have
    \begin{align}\label{eq:zoom-in}
        \int_{\SG}f\ \ d\mu=\sum_{w\in\{1,2,3\}^m}3^{-m}\int_\SG f\circ\psi_w\ \ d\mu.
    \end{align}
\end{enumerate}
Since we are interested in weak* convergence of $(\id_n)_{n\in\N}$, we first look at the average value of the functions $M_n(x,y,z)$ defined on $\SG_n$.

\begin{prop}\label{prop:average-of-m}
For $n\in\N$ and $x,y,z\in\N$, consider the function $M_n(x,y,z)$ on $\SG_n$ as recursively defined above. Then we have
    \begin{align*}
        \frac{1}{3^n}\sum_{v\in\SG_n}M_n(x,y,z)(v)=\frac{x+y+z}{3^n}+\frac{8}{3}\cdot\sum_{k=0}^{n-1}\frac{1}{3^k}=\frac{x+y+z-4}{3^n}+4.
    \end{align*}
\end{prop}
\begin{proof}
The proof is an easy induction argument. For $n=1$, a simple calculation shows that
\begin{align*}
    \frac{1}{3}\sum_{v\in\SG_1}M_1(x,y,z)(v)=\frac{1}{3}(3+3+2+x+y+z)
    =\frac{x+y+z}{3}+\frac{8}{3}=\frac{x+y+z-4}{3}+4.
\end{align*}
For the inductive step, assume that the claim holds for all $k<n$. Then we have
\begin{align*}
    \frac{1}{3^n}\sum_{v\in\SG_n}M_n(x,y,z)(v)&=\frac{x+y+z+8}{3^n}+3\cdot\frac{1}{3^{n}}\sum_{v\in\SG_{n-1}}M_{n-1}(0,0,0)(v)\\
    &=\frac{x+y+z+8}{3^n}+ \frac{1}{3^{n-1}}\sum_{v\in\SG_{n-1}}M_{n-1}(0,0,0)(v)\\
    &=\frac{x+y+z+8}{3^n}+\frac{8}{3}\cdot\sum_{k=0}^{n-2}\frac{1}{3^k}=\frac{x+y+z}{3^n}+\frac{8}{3}\cdot\sum_{k=0}^{n-1}\frac{1}{3^k}\\
    &=\frac{x+y+z-4}{3^n}+4,
\end{align*}
where the third equality above follows from the induction hypothesis.
\end{proof}
We next show that the average of the piecewise constant continuation
$\idel_n^{\constcont}$ converges to  $4$.
\begin{prop}\label{prop:average-of-id}
For the piecewise constant continuation $ \idel_n^{\constcont}:\SG\to\R$ of $\id_n:\SG_n\to\R$ we have 
    \begin{align*}
        \lim_{n\rightarrow\infty}\int_\SG \idel_n^{\constcont} \ \ d\mu = 4.
    \end{align*}
\end{prop}
\begin{proof}
Notice that for all $i\in\{1,2,3\}$ we have
\begin{align*}
    \int_\SG \idel_n^{\constcont} \circ \psi_i \ d\mu&=\frac{1}{3^{n-1}}\sum_{v\in\SG_{n-1}}M_{n-1}(2,2,2)(v)
    =\frac{6-4}{3^{n-1}}+4=\frac{2}{3^{n-1}}+4,
\end{align*}
where the first equality follows from the definition of $\idel_n$ as well as the rotational invariance of $\mu$ and the second equality follows from Proposition \ref{prop:average-of-m}. Thus in view of equation \eqref{eq:zoom-in}
\begin{align*}
    \int_\SG \idel_n^{\constcont}  \ d\mu=&\frac{1}{3}\sum_{i=1}^3\int_\SG \idel_n^{\constcont} \circ \psi_i \ d\mu =\frac{1}{3}\cdot3\cdot \frac{1}{3^{n-1}}\sum_{v\in\SG_{n-1}}M_{n-1}(2,2,2)(v)\\
    &=\frac{2}{3^{n-1}}+4\xrightarrow{n\rightarrow\infty}4.
\end{align*}
\end{proof}
Next, we consider the integrals over the sets $\psi_1(\SG),\psi_2(\SG)$ and $\psi_3(\SG)$.
\begin{prop}\label{prop:average-of-id-height1}
    For all $i\in\{1,2,3\}$ we have
    \begin{align*}
    \lim_{n\rightarrow\infty}\int_{\psi_i(\SG)} \idel_n^{\constcont} \ d\mu =\frac{4}{3}.
    \end{align*}
\end{prop}
\begin{proof}
For $i\in\{1,2,3\}$ we have
    \begin{align*}
        \int_{\psi_i(\SG)} \idel_n^{\constcont} \ d\mu=&\frac{1}{3}\int_{\SG}\idel_n^{\constcont}\circ\psi_i \ d\mu
        =\frac{3^{-n+1}}{3}\sum_{v\in\SG_{n-1}}M_{n-1}(2,2,2)(v)\\
    & =\frac{2}{3^n}+\frac{4}{3}\xrightarrow{n\rightarrow\infty}\frac{4}{3}
    \end{align*}
which proves the claim.    
\end{proof}
We can now prove the main result of the paper. Recall that $M_n^+$ is the configuration we obtain on $\SG_n$ from $M_n$ when rotating $\SG_n$ counterclockwise $120^\circ$ and $M_n^-$ is the configuration we obtain from a $120^\circ$ clockwise rotation of $M_n$, and $M_n^{\id}$ is simply $M_n$.
\begin{proof}[Proof of Theorem \ref{thm:scaling-limit-id}]
We prove \eqref{eq:conv-in-mean} for sets of the form $C=\psi_w(\SG)$ for $w\in\cup_{m=1}^\infty \{1,2,3\}^m$. The case $|w|=0$ follows from Proposition \ref{prop:average-of-id} and the case $|w|=1$ follows from Proposition \ref{prop:average-of-id-height1}, so we can assume $|w|>1$ and let us take some $n>|w|$. Then for $\alpha\in\{+,-,\id\}$ we have
\begin{align*}
\idel_n^{\constcont}\circ\psi_w = M_{n-|w|}^\alpha(x_n,y_n,z_n),
\end{align*}
for boundary values $x_n,y_n,z_n\leq3$, and this yields
    \begin{align*}
        \int_{\psi_w(\SG)} \idel_n^{\constcont} \ d\mu=&\frac{1}{3^{|w|}}\int_\SG \idel_n^{\constcont}\circ\psi_w \ d\mu
        =\frac{1}{3^{|w|+n-|w|}}\sum_{v\in\SG_{n-|w|}}M_{n-|w|}(x_n,y_n,z_n)(v)\\
        =&\frac{x_n+y_n+z_n-4}{3^{n}}+\frac{8}{2\cdot3^{|w|}}(1-3^{|w|-n})
        \\&\xrightarrow{n\rightarrow\infty}\frac{8}{2\cdot 3^{|w|}}= \int_{\psi_w(\SG)}4\ d\mu=4\mu(\psi_w(\SG)).\\
    \end{align*}
which together with \eqref{eq:char-of-weak*}) gives
\begin{align*}
    \idel_n^{\constcont}\xrightarrow{w*}4,
\end{align*}
and this completes the proof of Theorem \ref{thm:scaling-limit-id}.
\end{proof}

\section{Generalizations}
In Section \ref{sec:id-scaling}, we have proven that the identity elements on $\SG_n$ with normal boundary conditions admit a weak* scaling limit on $\SG$. What about the identity elements with different boundary conditions? Do they also admit a weak* scaling limit? It is easy to find the weak* scaling limit for other boundaries, for example for sinked boundary by going through the calculations of Proposition \ref{prop:average-of-id}, \ref{prop:average-of-id-height1} and the proof of Theorem \ref{thm:scaling-limit-id}. We prove below that a class of functions - under which the known identity elements with different boundary conditions fall - all admit scaling limits.
\begin{prop}
    Let $a,b,c\in\R$. Define the function $f_1(x,y,z):\SG_1\to\R$ by setting the values at the vertices as follows
\begin{center}
\begin{tikzpicture}[baseline=7ex, scale=0.7]
    \node[shape=circle,draw=black] (A) at (0,0) {x};
    \node[shape=circle,draw=black] (B) at (2,0) {a};
    \node[shape=circle,draw=black] (C) at (4,0) {y};
    \node[shape=circle,draw=black] (D) at (1,1.73) {b};
    \node[shape=circle,draw=black] (E) at (3,1.73) {c};
    \node[shape=circle,draw=black] (F) at (2,1.73*2) {z} ;
    
    \path [-] (A) edge node[left] {} (B);
    \path [-] (A) edge node[left] {} (D);
    \path [-] (B) edge node[left] {} (D);
    \path [-] (B) edge node[left] {} (C);
    \path [-] (E) edge node[left] {} (B);
    \path [-] (E) edge node[left] {} (C);
    \path [-] (F) edge node[left] {} (E);
    \path [-] (F) edge node[left] {} (D);
    \path [-] (E) edge node[left] {} (D);
\end{tikzpicture}=
\begin{tikzpicture}[baseline=4.5ex]
    \node[shape=circle,draw=black] (A) at (0,0) {x};
    \node[shape=circle,draw=black] (B) at (2,0) {y};
    \node[shape=circle,draw=black] (D) at (1,1.73) {z};
    
    \path [-] (A) edge node[left] {} (B);
    \path [-] (A) edge node[left] {} (D);
    \path [-] (B) edge node[left] {} (D);
    
    \node[] at (1,1.73/3) {$f_1$};
\end{tikzpicture}=$\quad f_1(x,y,z)$.
\end{center}
We inductively define $f_n(x,y,z):\SG_n\to\R$ from $f_{n-1}(x,y,z)$ as
\begin{center}
\begin{tikzpicture}[baseline=6ex, scale=0.7]
    \node[shape=circle,draw=black] (A) at (0,0) {x};
    \node[shape=circle,draw=black] (B) at (2,0) {a};
    \node[shape=circle,draw=black] (C) at (4,0) {y};
    \node[shape=circle,draw=black] (D) at (1,1.73) {b};
    \node[shape=circle,draw=black] (E) at (3,1.73) {c};
    \node[shape=circle,draw=black] (F) at (2,1.73*2) {z} ;
    
    \path [-] (A) edge node[left] {} (B);
    \path [-] (A) edge node[left] {} (D);
    \path [-] (B) edge node[left] {} (D);
    \path [-] (B) edge node[left] {} (C);
    \path [-] (E) edge node[left] {} (B);
    \path [-] (E) edge node[left] {} (C);
    \path [-] (F) edge node[left] {} (E);
    \path [-] (F) edge node[left] {} (D);
    \path [-] (E) edge node[left] {} (D);
    
    \node[] at (1,1.73/3) {$f_{n-1}$};
    \node[] at (3,1.73/3) {$f_{n-1}$};
    \node[] at (2,1.73/3+1.73) {$f_{n-1}$};
\end{tikzpicture}=
\begin{tikzpicture}[baseline=4ex]
    \node[shape=circle,draw=black] (A) at (0,0) {x};
    \node[shape=circle,draw=black] (B) at (2,0) {y};
    \node[shape=circle,draw=black] (D) at (1,1.73) {z};
    
    \path [-] (A) edge node[left] {} (B);
    \path [-] (A) edge node[left] {} (D);
    \path [-] (B) edge node[left] {} (D);
    
    \node[] at (1,1.73/3) {$f_{n}$};
\end{tikzpicture}=$\quad f_{n}(x,y,z)$.
\end{center}
Then the sequence $(f_n^{\constcont}(x,y,z))_{n\in\N}$ of piecewise constant continuations of $(f_n(x,y,z))_{n\in\N}$ converges in the weak* sense to the constant function on $ \SG$ with value $\frac{a+b+c}{2}$, that is
\begin{align*}
    f_n^{\constcont}(x,y,z)\xrightarrow{w*}\frac{a+b+c}{2}.
\end{align*}
\end{prop}
\begin{proof}
 Let $a,b,c,x,y,z\in\R$ and set $N=\max\{|a|,|b|,|c|,|x|,|y|,|z|\}$. Furthermore choose $w\in\cup_{m=1}^\infty \{1,2,3\}^m$. We then have for some $x_n,y_n,z_n\in[-N,N]$
    \begin{align*}
        \int_{\psi_w(\SG)}f_n^{\constcont}(x,y,z)\ d\mu=\frac{1}{3^{n}}\sum_{v\in\SG_{n-|w|}}f_{n-|w|}(x_n,y_n,z_n)(v).
    \end{align*}
Similar to the proof of Proposition \ref{prop:average-of-m}, we can use the inductive definition of the sequence $(f_n^{\constcont}(x,y,z))$ to obtain
\begin{align*}
    \frac{1}{3^{n}}\sum_{v\in\SG_{n-|w|}}f_{n-|w|}(x_n,y_n,z_n)(v)&=\frac{x_n+y_n+z_n}{3^n}+\frac{a+b+c}{3^{|w|+1}}\sum_{k=0}^{n-|w|-1}3^{-k}\\
    &=\frac{x_n+y_n+z_n}{3^n}+\frac{a+b+c}{2\cdot3^{|w|}}(1-3^{|w|-n})
\end{align*}
Combining all of the above we get
\begin{align*}
    \int_{\psi_w(\SG)}f_n^{\constcont}(x,y,z)\ d\mu\xrightarrow{n\rightarrow\infty}\frac{a+b+c}{2\cdot3^{|w|}},
\end{align*}
completing the proof.
\end{proof}
Finally, we can also combine three function families with weak* scaling limits on $\SG$ to again obtain a converging sequence.
\begin{prop}
Let $(f_n:\SG_n\rightarrow\R)_{n\in\N},(g_n:\SG_n\rightarrow\R)_{n\in\N},(h_n:\SG_n\rightarrow\R)_{n\in\N}$ be sequences of functions with weak* scaling limits on $\SG$ given by $f,g,h.$ Furthermore, let $\alpha,\beta,\gamma$ be from the set $\{+,-,\id\}$. Let us define the new sequence of functions $\iota_n:\SG_n\rightarrow\R$ given by
\begin{center}
\begin{tikzpicture}[baseline=6ex, scale=0.7]
    \node[shape=circle,draw=black] (A) at (0,0) {};
    \node[shape=circle,draw=black] (B) at (2,0) {};
    \node[shape=circle,draw=black] (C) at (4,0) {};
    \node[shape=circle,draw=black] (D) at (1,1.73) {};
    \node[shape=circle,draw=black] (E) at (3,1.73) {};
    \node[shape=circle,draw=black] (F) at (2,1.73*2) {} ;
    
    \path [-] (A) edge node[left] {} (B);
    \path [-] (A) edge node[left] {} (D);
    \path [-] (B) edge node[left] {} (D);
    \path [-] (B) edge node[left] {} (C);
    \path [-] (E) edge node[left] {} (B);
    \path [-] (E) edge node[left] {} (C);
    \path [-] (F) edge node[left] {} (E);
    \path [-] (F) edge node[left] {} (D);
    \path [-] (E) edge node[left] {} (D);
    
    \node[] at (1,1.73/3) {$f_{n-1}^{\alpha}$};
    \node[] at (3,1.73/3) {$g_{n-1}^{\beta}$};
    \node[] at (2,1.73/3+1.73) {$h_{n-1}^{\gamma}$};
\end{tikzpicture}=
\begin{tikzpicture}[baseline=4ex]
    \node[shape=circle,draw=black] (A) at (0,0) {};
    \node[shape=circle,draw=black] (B) at (2,0) {};
    \node[shape=circle,draw=black] (D) at (1,1.73) {};
    
    \path [-] (A) edge node[left] {} (B);
    \path [-] (A) edge node[left] {} (D);
    \path [-] (B) edge node[left] {} (D);
    
    \node[] at (1,1.73/3) {$\iota_{n}$};
\end{tikzpicture},
\end{center}
where at the cut points $\iota_n$ takes as value the sum of the corresponding functions meeting at the cut point. Then the sequence $(\iota_n^{\constcont}:\SG\rightarrow\R)$ has a weak* limit on $\SG$ given by
\begin{center}
$\iota_n^{\constcont}\xrightarrow{w*}$
\begin{tikzpicture}[baseline=6ex, scale=0.7]
    \node[shape=circle,draw=black] (A) at (0,0) {};
    \node[shape=circle,draw=black] (B) at (2,0) {};
    \node[shape=circle,draw=black] (C) at (4,0) {};
    \node[shape=circle,draw=black] (D) at (1,1.73) {};
    \node[shape=circle,draw=black] (E) at (3,1.73) {};
    \node[shape=circle,draw=black] (F) at (2,1.73*2) {} ;
    
    \path [-] (A) edge node[left] {} (B);
    \path [-] (A) edge node[left] {} (D);
    \path [-] (B) edge node[left] {} (D);
    \path [-] (B) edge node[left] {} (C);
    \path [-] (E) edge node[left] {} (B);
    \path [-] (E) edge node[left] {} (C);
    \path [-] (F) edge node[left] {} (E);
    \path [-] (F) edge node[left] {} (D);
    \path [-] (E) edge node[left] {} (D);
    
    \node[] at (1,1.73/3) {$f^{\alpha}$};
    \node[] at (3,1.73/3) {$g^{\beta}$};
    \node[] at (2,1.73/3+1.73) {$h^{\gamma}$};
\end{tikzpicture}$=\iota$.
\end{center}
\end{prop}
\begin{proof}
    Choose $w\in\cup_{m=1}^\infty \{1,2,3\}^m$. If $|w|=0$, we have
    \begin{align*}
        \int_\SG \iota_n^{\constcont} \ d\mu=&\frac{1}{3}\Big(\int_\SG f_{n-1}^{\alpha} \ d\mu+\int_\SG g_{n-1}^{\beta} \ d\mu+\int_\SG h_{n-1}^{\gamma} \ d\mu\Big)\\
        \xrightarrow{n\rightarrow\infty}&\frac{1}{3}\Big(\int_\SG f^{\alpha} \ d\mu+\int_\SG g^{\beta} \ d\mu+\int_\SG h^{\gamma} \ d\mu\Big)
        =\int_\SG \iota \ d\mu.
    \end{align*}
Now let us assume that $|w|\geq 1$. If $w_1=1$ and we write $w'=(w_2,w_3,...,w_n)$, then we have
    \begin{align*}
        \int_{\psi_w(\SG)}\iota_n^{\constcont}\ d\mu=&\frac{1}{3}\int_{\psi_{w'}(\SG)}(f_n^\alpha)^{\constcont}\ d\mu
        \xrightarrow{n\rightarrow\infty}\frac{1}{3}\int_{\psi_{w'}(\SG)}f^\alpha \ d\mu
        =\int_{\psi_w(\SG)}\iota \ d\mu.
    \end{align*}
    The cases for $w_1=2$ and $w_1=3$ work the same, completing the proof.
\end{proof}
As a consequence, Section $\ref{sec:id-scaling}$ can be generalized for the scaling limit of the identity elements in the sandpile group when choosing the top corner as sink vertex, or when choosing any two out of the three corner vertices as sink vertices, as described in \cite{chen-sandpile-limit-shape}.
\begin{cor}
For $n\in\N$, denote by $\idel_n(1)$ the identity element of the sandpile group on $\SG_n$ when choosing the top corner as the sink vertex and by $\idel_n(2)$ the identity element when choosing both the top and lower right corner as sinks. Then we have as weak* scaling limits
\begin{align*}
\idel_n(1)^{\constcont}\xrightarrow{w*} 3\quad \text{and} \quad \idel_n(2)^{\constcont}\xrightarrow{w*}4.
\end{align*}
\end{cor}

\paragraph{Remarks.} We remark that the weak* star convergence seems to be the right notion of convergence for the rescaled sandpiles, since
the sandpiles consist of oscillating patterns of fractal nature, thus no kind of pointwise convergence can be expected to hold. Instead, the sandpiles converge in the sense that the regions where the sandpile oscillates converge to their average value, and this is exactly what the weak* convergence captures. This type of limit washes out the structure present in the identity element. While we didn't try to capture the structure of the identity element by taking a different type of limit, this might be a very interesting problem, as one of the referees points out.

\paragraph{Funding.} The research of Robin Kaiser and Ecaterina Sava-Huss is supported by the Austrian Science Fund (FWF): P 34129. We are very grateful to the three referees for a very careful reading
of the paper and for the valuable comments and suggestions.

\bibliographystyle{alpha}
\bibliography{lit}

\textsc{Robin Kaiser}, Universität Innsbruck, Institut für Mathematik, Technikerstraße 13, 6020 Innsbruck, Austria.
\texttt{Robin.Kaiser@uibk.ac.at}

\textsc{Ecaterina Sava-Huss}, Universität Innsbruck, Institut für Mathematik, Technikerstraße 13, 6020 Innsbruck, Austria.
\texttt{Ecaterina.Sava-Huss@uibk.ac.at}
\end{document}